\documentclass[a4paper,12pt]{amsart}
\usepackage{graphicx}
\usepackage{amsmath,amssymb,amsthm,amsfonts}
\usepackage{amssymb}
\usepackage[active]{srcltx}

\newtheorem{thm}{Theorem}[section]

\newtheorem{lem}[thm]{Lemma}
\newtheorem{prop}[thm]{Proposition}
\newtheorem{rem}[thm]{Remark}
\newtheorem{defn}[thm]{Definition}

\newcommand{\bremark}{\begin{rem} \textup}
\newcommand{\eremark}{\end{rem} }

\newcommand{\cuad}{{\sqcap\kern-.68em\sqcup}}

\newcommand{\R}{{\mathbb{R}}}
\newcommand{\N}{{\mathbb{N}}}

\renewcommand{\rho}{\varrho}
\renewcommand{\theta}{\vartheta}

\begin{document}

\parindent 0pc
\parskip 6pt
\overfullrule=0pt

\title[On the H\"{o}pf Boundary Lemma]{On the  H\"{o}pf Boundary Lemma for quasilinear problems involving singular nonlinearities and applications}

\author{Francesco Esposito* and Berardino Sciunzi*}

\date{\today}

\date{\today}

\address{* Dipartimento di Matematica e Informatica, UNICAL,
Ponte Pietro  Bucci 31B, 87036 Arcavacata di Rende, Cosenza, Italy.}

\email{esposito@mat.unical.it}

\email{sciunzi@mat.unical.it}

\maketitle

\date{\today}

\begin{abstract}
In this paper we consider positive solutions to quasilinear elliptic
problem with singular nonlinearities.  We provide a H\"{o}pf type
boundary lemma via a suitable scaling argument that allows to deal
with the lack of regularity of the solutions up to the boundary.
\end{abstract}

\section{Introduction}
We deal with positive weak solutions to the singular quasilinear
elliptic problem:
\begin{equation}\tag{$\mathcal P$}\label{problem1}
\begin{cases} \displaystyle
-\Delta_p u =\frac{1}{u^\gamma} + f(u) & \text{in}\,\,\Omega  \\
u> 0 &  \text{in}\,\,\Omega  \\
u=0 &  \text{on}\,\,\partial \Omega
\end{cases}
\end{equation}
where $p>1$, $\gamma > 1$, $\Omega$ is a $C^{2,\alpha}$ bounded
domain of $\R^N$ with $N \geq 1$ and $f:\Omega \rightarrow \R$
locally Lipschitz continuous.  A key point to have in mind in the study of semilinear or quasilinear problems involving singular nonlinearities is the fact that the source term loses regularity at zero, namely the problem is singular near the boundary. As a first consequence, solutions are not smooth up to the boundary (see \cite{lazer}) and the gradient generally blows up near the boundary in such a way that $u\notin W^{1,p}_0(\Omega)$. Therefore, here and in all the paper, we mean that $u\in C^{1,\alpha}(\Omega)$ is a solution to \eqref{problem1} in the weak distributional meaning according to Definition \ref{weakformbounded}. Existence and uniqueness results regarding problem \eqref{problem1} can be found e.g. in \cite{arcoya,boccardo2,boccardo,masha,montrucchio,nodt,giachetti,oliva1,oliva2}.
\\

\noindent In this setting we prove here a H\"{o}pf type boundary lemma regarding the sign of the derivatives of the solution near the boundary and in the interior of the domain. 
To state our result we need some notation thus we shall denote with
 $I_\delta(\partial \Omega)$  a neighborhood of the boundary with the \emph{unique nearest point property} (see e.g.
\cite{FOO}). For $x \in I_\delta(\partial \Omega)$ we denote by $\hat{x} \in
\partial \Omega$ the point such that $|x-\hat{x}| = \text{dist}(x,\partial
\Omega)$ and we set:
\begin{equation}\label{normin}
\eta(x):=\frac{x-\hat{x}}{|x-\hat{x}|}.
\end{equation}
With this notation we have the following:
\begin{thm} \label{Hopf}
Let $u \in C^{1,\alpha}(\Omega) \cap C(\overline{\Omega})$ be
a positive solution to \eqref{problem1}. Then, for any $\beta>0$,  there exists a
neighborhood $I_\delta(\partial \Omega)$ of $\partial
\Omega$, such that
\begin{equation}\label{derhopf}
\partial_{\nu(x)} u > 0 \qquad \quad \forall \, x \in I_\delta(\partial
\Omega)
\end{equation}
whenever  $\nu(x) \in \R^N$ with $\|\nu(x)\|=1$ and $( \nu(x) , \eta(x) ) \geq \beta$.
\end{thm}
The technique of  E. H\"opf \cite{hopf} (see also \cite{GT}) has been already developed, and improved, in the quasilinear setting. We refer the readers to \cite{pucser} (see also \cite{V}). Nevertheless the proof of Theorem \ref{normin}, namely the proof of the lemma in the case when it appears the singular term $u^{-\gamma}$, cannot be carried out in the standard way mainly because the solutions are not of class $C^1$ up to the boundary. More precisely the proofs in \cite{hopf,GT,pucser,V} has the common feature of basing on the comparison of the solution  with subsolutions that have a known behaviour on the boundary.
This approach, with some difficulty to take into account, can be exploited also in the singular case since $u^{-\gamma}$ has the right monotonicity behaviour. This actually leads to  control  the behaviour of the solution near the boundary with a comparison based on the distance function. This is in fact also behind our Theorem \ref{estbelowabove}. Although some of the underlying ideas in our approach  have a common flavour with the ones exploited in \cite{CES}, the proofs that we exploit are new and adapted to the degenerate nonlinear nature of the $p$-laplacian.
\\

\noindent We are mainly concerned with the study of the sign of the derivatives near the boundary.
Such a control  is generally deduced a posteriori, by contradiction, assuming that the solution is $C^1$ up to the boundary. In our setting this is not a natural assumption and we develop a different technique that in any case exploits very basic arguments of common use. In fact we carry out a scaling argument near the boundary that leads to a limiting problem in the half space.
 \begin{equation}\tag{$\mathcal P_1$}\label{problem}
\begin{cases} \displaystyle
-\Delta_p u= \frac{1}{u^\gamma} & \text{in}\,\,\R^N_+  \\
u> 0 &  \text{in}\,\,\R^N_+  \\
u=0 &  \text{on}\,\,\partial \R^N_+
\end{cases}
\end{equation}
where $p>1$, $\gamma > 1$,  $N \geq 1$, $\R^N_+:=\{x = (x_1,...,x_N)
\in \R^N \ | \ x_N>0 \}$ and $u \in C^{1,\alpha}(\R^N_+) \cap
C(\overline{\R^N_+})$.

Our scaling argument leads in fact to the study of a limiting profile which
is a solution to \eqref{problem} and obeys to suitable decay
assumptions. It is therefore crucial for our technique, and may also
have an independent interest, the following classification result:

\begin{thm}\label{uniqueness}
Let $\gamma > 1$ and let $u \in C^{1,\alpha}(\R^N_+) \cap
C(\overline{\R^N_+})$ be a solution to problem \eqref{problem} such
that
\begin{equation}\label{grow223324}
\underline{c} x_N^\beta \leq u(x) \leq \overline{C} x_N^\beta  \qquad\text{with}\quad\beta:=\frac{p}{\gamma+p-1}
\end{equation}
and
$\underline{c}, \overline{C}\in\mathbb{R}^+$. Then
\begin{equation} \label{unboundedSOl}
u(x)=u(x_N)=M x_N^\beta \qquad\text{with}\quad M:=\left[
\frac{(\gamma+p-1)^p}{p^{p-1}(p-1)(\gamma-1)}\right]^{\frac{1}{\gamma+p-1}}\,.
\end{equation}
\end{thm}

The H\"{o}pf boundary lemma is a fundamental tool in many applications. We exploit it here to develop the \emph{moving plane method} (see \cite{A,BN,GNN,serrin}) for problem \eqref{problem1} obtaining the following:

\begin{thm}\label{symmetrybdd}
Let $\Omega$ be a bounded smooth
domain of $\R^N$ which is strictly convex in the
$x_1$-direction and symmetric with respect to the hyperplane $\{x_1
= 0\}$. Let $u \in C^{1,\alpha}(\Omega) \cap
C(\overline{\Omega})$ be a positive solution of problem
\eqref{problem1} with $f(s)>0$ for $s>0$ ($f(0)\geq 0$). Then it follows that u is symmetric with respect
to the hyperplane $\{x_1 = 0\}$ and increasing in the
$x_1$-direction in $\Omega \cap \{x_1 < 0\}$.\\
\noindent In particular if the domain is a ball, then the solution is radial and radially decreasing.
\end{thm}

The proof is based on the fact that the H\"{o}pf type boundary lemma (Theorem \ref{Hopf}) allows to exploit and improve the results in \cite{pace,DSJDE}.
The key point is the fact that the monotonicity of the solution near the boundary is provided by the H\"{o}pf  lemma and the moving plane procedure can be exploited working in the interior of the domain where the nonlinearity is no more singular. In the semilinear case, a similar result is proved in \cite{masha,montrucchio}.

For the reader's convenience we sketch the proofs here below.
\begin{itemize}
\item[-] In Section \ref{sec2} we prove 1D-symmetry result in half spaces for problem \eqref{problem}, see Theorem \ref{1D}. Mainly we develop a comparison principle to compare the solution $u$ and it's translation $u_\tau:=u(x-\tau e_1)$. Even if the source term is decreasing, a quite technical approach is needed because the operator is nonlinear and we are reduced to work in unbounded domains. The 1D-symmetry result obtained  leads us to the study of a one dimensional problem in $\mathbb R^+$. We carry out this analysis proving a uniqueness result (see Proposition \ref{ODEuniqueness}) that provides, as a corollary, the proof of Theorem \ref{uniqueness}.

\item[-] Section \ref{sec3} is the core of the paper. We prove here Theorem \ref{Hopf} developing the scaling argument that leads to the problem in the half space. To this aim we strongly exploit the asymptotic estimates deduced in Theorem \ref{estbelowabove}. The proof follows by contradiction thanks to the classification result Theorem \ref{uniqueness}.

\item[-] Finally, in Section \ref{sec4}, we apply our H\"{o}pf type boundary lemma to prove the symmetry and monotonicity result stated in Theorem \ref{symmetrybdd}. The proof is based on the joint use of the \emph{moving plane method} and the monotonicity information near the boundary provided by Theorem \ref{Hopf} that allow to avoid the region where the problem is singular.

\item[-] In the Appendix we prove Lemma \ref{derbehaviorODE}, that is a very useful tool in the ODE analysis. Moreover we run through again, in the quasilinear setting, the technique of \cite{lazer} to provide asymptotic estimates for the solutions near the boundary in terms of the distance function, see Theorem \ref{estbelowabove}. Finally we prove Lemma \ref{wcpbd} that is a weak comparison principle in bounded domain that we used in the proof of Theorem \ref{estbelowabove}.
\end{itemize}

\section{1D symmetry in the half space, ODE analysis and proof of Theorem \ref{uniqueness}}\label{sec2}
Solutions to $p$-Laplace equations are generally of class $C^{1,\alpha}$, see \cite{DB,T}.
Therefore a solution to \eqref{problem1} has
to be understood in the weak distributional meaning  taking into account the singular nonlinearity. We state the following:
\begin{defn}\label{weakformbounded}  We say that $u \in W^{1,p}_{loc}(\Omega) \cap C(\overline\Omega)$ is a (positive) weak
solution to problem \eqref{problem1} if
\begin{equation}\label{problem1weaksol}
\int_{\Omega} |\nabla u|^{p-2}(\nabla u, \nabla \varphi) \, dx \, =
\int_{\Omega} \frac{ \varphi}{u^\gamma} \, dx \, + \int_\Omega f(u)
\varphi \, dx \qquad \forall \varphi \in C_c^\infty(\Omega).
\end{equation}
 We say that
$u \in W^{1,p}_{loc}(\Omega) \cap C(\overline\Omega)$ is a weak
\emph{subsolution} of problem \eqref{problem1} if
\begin{equation}\label{problem1weaksubsol}
\int_{\Omega} |\nabla u|^{p-2}(\nabla u, \nabla \varphi) \, dx \,
\leq \int_{\Omega} \frac{\varphi}{u^\gamma}  \, dx \, + \int_\Omega
f(u) \varphi \, dx \qquad \forall \varphi \in C_c^\infty(\Omega), \,
\varphi \geq 0.
\end{equation}
Similarly, we say that $u \in W^{1,p}_{loc}(\Omega) \cap C(\overline\Omega)$ is a weak \emph{supersolution} of problem
\eqref{problem1}  if
\begin{equation}\label{problem1weaksupersol}
\int_{\Omega} |\nabla u|^{p-2} (\nabla u, \nabla \varphi) \, dx \,
\geq \int_{\Omega} \frac{\varphi}{u^\gamma} \, dx \, + \int_\Omega
f(u) \varphi \, dx \qquad \forall \varphi \in C_c^\infty(\Omega), \,
\varphi \geq 0.
\end{equation}
\end{defn}

In the following we further use the following  inequalities:
$\forall \eta, \eta' \in \mathbb{R}^N$ with $|\eta|+|\eta'|>0$ there
exists positive constants $C_1, C_2,C_3,C_4$ depending on $p$ such
that
\begin{equation}\label{eq:inequalities}
\begin{split}
[|\eta|^{p-2}\eta-|\eta'|^{p-2}\eta'][\eta- \eta'] &\geq C_1
(|\eta|+|\eta'|)^{p-2}|\eta-\eta'|^2, \\ \\
\|\eta|^{p-2}\eta-|\eta'|^{p-2}\eta '| & \leq C_2
(|\eta|+|\eta'|)^{p-2}|\eta-\eta '|,\\\\
[|\eta|^{p-2}\eta-|\eta'|^{p-2}\eta '][\eta-\eta '] & \geq C_3
|\eta-\eta '|^p \quad \quad\mbox{if}\quad p\geq 2,\\\\
\|\eta|^{p-2}\eta-|\eta'|^{p-2}\eta '| & \leq C_4 |\eta-\eta
'|^{p-1} \quad\mbox{if}\quad 1 < p \leq 2.
\end{split}
\end{equation}

\begin{thm}\label{1D}
Let $\gamma > 1$ and let $u \in C^{1,\alpha}(\R^N_+) \cap
C(\overline{\R^N_+})$ be a solution to problem \eqref{problem} such
that
\begin{equation}\label{grow2233245555}
\underline{c} x_N^\beta \leq u(x) \leq \overline{C} x_N^\beta  \quad
\forall x \in \R_+^N
 \end{equation}with $\displaystyle
\beta:=\frac{p}{\gamma+p-1}$. Then
\begin{equation}\label{1deq}
u_{x_i} \equiv 0
\end{equation}
for every $i=1,\dots,N-1$. Namely, $u(x)=u(x_N).$
\end{thm}

\begin{proof}[Proof of Theorem \ref{1D}] We start with a gradient estimate showing that
	\begin{equation}\label{gradientestimate}
		|\nabla u(x)| \leq
		\frac{C_b}{x_N^\frac{\gamma-1}{\gamma+p-1}}\,.
	\end{equation}
	To prove this fact we use the notation $x = (x_1,...,x_N)=(x',x_N)
	\in \R^N$ and, with no loss of generality
	we  consider a point
	$P_c:=(0',x_N^c)$.  Setting
	$$w(x):=\frac{u(x_N^c\cdot x)}{(x_N^c)^\beta}$$
	it follows that
	\begin{equation}\label{scalinggradient}
		-\Delta_p w = \frac{1}{w^\gamma}\qquad  \text{in}\,\,\R^N_+\,.
	\end{equation}
	We restrict our attention to the problem
	\begin{equation}\label{diben}
		\begin{cases} \displaystyle
			-\Delta_p w= \frac{1}{w^\gamma} & \text{in}\,\, B_{\frac{1}{2}}(0',1)  \\
			w> 0 &  \text{in}\,\, B_{\frac{1}{2}}(0',1)
		\end{cases}
	\end{equation}
	so that, by \eqref{grow2233245555}, it follows that $w$ is bounded and  $\frac{1}{w^\gamma} \in
	L^\infty(B_{\frac{1}{2}}(0',1))$. Therefore,  by standard $C^{1,\alpha}$ estimates
	\cite{DB,T}, we deduce  that
	$$\|w\|_{C^1(B_{\frac{1}{4}}(0',1))} \leq
	C_b.$$
	Scaling back we get \eqref{gradientestimate}.

	Arguing by contradiction, without loss of generality, we assume that
	there exists $P_0 \in \R^N$ such that $u_{x_1}(P_0) > 0$. Hence
	there exists $\delta>0$ sufficiently small such that $u_{x_1}(x)>0$ for all $x \in B_\delta(P_0)$. Now we define
	\begin{equation}\label{translation}
		u_\tau(x):=u(x-\tau e_1)
	\end{equation}
	where $0<\tau < \delta$. Hence by the Mean Value Theorem it follows
	that
	\begin{equation}\label{kgjklfgjfgjò}
		u(P_0)-u_\tau(P_0)=u_{x_1}(\xi) \tau > \hat{C} \tau > 0
	\end{equation}
	where $\xi \in \{ t P_0 + (1-t)(P_0-\tau e_1), t \in [0,1]  \}$.
	Moreover, there exists $k>0$ sufficiently large such that, by the
	Mean Value Theorem and \eqref{gradientestimate}, we have
	\begin{equation}\label{meanvalue}
		|u-u_\tau|\leq
		\frac{\check{C}\tau}{x_N^\frac{\gamma-1}{\gamma+p-1}} \quad
		\text{in} \; \R^N_+ \cap \{x_N \geq k\}.
	\end{equation}
	Now we set
	\begin{equation}\label{supremum}
		S:=\sup_{x \in \R^N_+} (u-u_\tau)>0.
	\end{equation}
	We also note that $S < + \infty$ by \eqref{grow2233245555} and
	\eqref{meanvalue}. Let us consider
	\begin{equation}\label{positivepart}
		w_{\tau,\varepsilon}(x):=[u-u_\tau-(S-\varepsilon)]^+
	\end{equation}
	for every $\varepsilon>0$ small enough. We notice that, by \eqref{grow2233245555} and
	\eqref{meanvalue},
	\begin{equation}\label{XXX}
		\text{supp}(w_{\tau, \varepsilon}) \subset
		\subset  \{\hat k\leq x_N \leq \hat{K}\}
	\end{equation}
	for some $\hat{k},\hat K>0$. We
	consider a standard cutoff function $\varphi_R:=\varphi_R(x')$ such
	that $\varphi_R=1$ in $B'_R(0)$, $\varphi_R=0$ in $(B'_{2R}(0))^c$
	and $|\nabla \varphi_R| \leq \frac{2}{R}$ in $B'_{2R}(0) \setminus
	B'_R(0)$, where $B'_R(0)$ denotes the $(N-1)$-dimensional ball of
	center $0$ and radius $R$.

	We distinguish two cases:
	
	\
	
	\noindent {\bf {Case 1: $1<p <2$.}} We set
	\begin{equation}\label{test}
		\psi:=w_{\tau, \varepsilon}^\alpha \varphi_R^2
	\end{equation}
	where $\alpha >0$, $w_{\tau,\varepsilon}$ is defined in
	\eqref{positivepart} and $\varphi_R$ is the  cutoff function
	defined here above. First of all we notice that $\psi$ belongs to
	$W^{1,p}_0(\R^N_+)$. By density argument we can take $\psi$ as test
	function in the weak formulation of problem \eqref{problem}, see
	Definition \ref{weakformbounded}, so that,  subtracting the equation for
	$u$ and $u_\tau$, we obtain
	\begin{equation}\label{starteqpminhi}
		\begin{split}
			&\alpha \int_{\R^N_+ \cap \text{supp}(\psi)} (|\nabla u|^{p-2}
			\nabla u - |\nabla u_\tau|^{p-2} \nabla u_\tau, \nabla w_{\tau,
				\varepsilon})
			w_{\tau, \varepsilon}^{\alpha-1} \varphi_R^2 \, dx \\
			=&-2 \int_{\R^N_+ \cap \text{supp}(\psi)} (|\nabla u|^{p-2} \nabla u
			- |\nabla u_\tau|^{p-2} \nabla u_\tau, \nabla \varphi_R) w_{\tau,
				\varepsilon}^\alpha \varphi_R \, dx\\
			& + \int_{\R^N_+ \cap \text{supp}(\psi)} \left( \frac{1}{u^\gamma}
			-\frac{1}{u_\tau^\gamma}\right) w_{\tau, \varepsilon}^\alpha
			\varphi_R^2 \, dx\,.
		\end{split}
	\end{equation}
	From \eqref{starteqpminhi}, using \eqref{eq:inequalities} and the
	Mean Value Theorem, we obtain
	\begin{equation}\label{middlediseq1}
		\begin{split}
			&\alpha C_1 \int_{\R^N_+ \cap \text{supp}(\psi)} \left(|\nabla u| +
			|\nabla u_\tau|\right)^{p-2} |\nabla w_{\tau, \varepsilon}|^2
			w_{\tau, \varepsilon}^{\alpha-1} \varphi_R^2 \, dx \\
			&\leq \alpha \int_{\R^N_+ \cap \text{supp}(\psi)} (|\nabla u|^{p-2}
			\nabla u - |\nabla u_\tau|^{p-2} \nabla u_\tau, \nabla w_{\tau,
				\varepsilon})
			w_{\tau, \varepsilon}^{\alpha-1} \varphi_R^2 \, dx\\
			&= -2 \int_{\R^N_+ \cap \text{supp}(\psi)} (|\nabla u|^{p-2} \nabla
			u - |\nabla u_\tau|^{p-2} \nabla u_\tau, \nabla \varphi_R) w_{\tau,
				\varepsilon}^\alpha \varphi_R \, dx\\
			&+ \int_{\R^N_+ \cap \text{supp}(\psi)} \left( \frac{1}{u^\gamma} -
			\frac{1}{u_\tau^\gamma}\right) w_{\tau, \varepsilon}^\alpha
			\varphi_R^2 \, dx\\
			&\leq 2C_4 \int_{\R^N_+ \cap \text{supp}(\psi)} |\nabla
			(u-u_\tau)|^{p-1} |\nabla \varphi_R| w_{\tau,\varepsilon}^\alpha
			\varphi_R \, dx - \gamma \int_{\R^N_+ \cap \text{supp}(\psi)}
			\frac{1}{\xi^{\gamma+1}} (u-u_\tau)w_{\tau,\varepsilon}^\alpha
			\varphi_R^2 \, dx
		\end{split}
	\end{equation}
	where $\xi$ belongs to $\overrightarrow{(u,u_\tau)}$. Hence, recalling also \eqref{gradientestimate}, we deduce that
	\begin{equation}\label{middlediseq2}
		\begin{split}
			&\alpha C_1 \int_{\R^N_+ \cap \text{supp}(\psi)} \left(|\nabla u| +
			|\nabla u_\tau|\right)^{p-2} |\nabla w_{\tau, \varepsilon}|^2
			w_{\tau, \varepsilon}^{\alpha-1} \varphi_R^2 \, dx\\
			&\leq 2C_4 \int_{\R^N_+ \cap \text{supp}(\psi)} |\nabla
			(u-u_\tau)|^{p-1} |\nabla \varphi_R| w_{\tau,\varepsilon}^\alpha
			\varphi_R \, dx - \gamma \int_{\R^N_+ \cap \text{supp}(\psi)}
			\frac{1}{\xi^{\gamma+1}} w_{\tau,\varepsilon}^{\alpha+1} \varphi_R^2
			\, dx \\
			& \leq \check{C}\int_{\R^N_+ \cap \text{supp}(\psi)}  |\nabla
			\varphi_R| w_{\tau,\varepsilon}^\alpha  \, dx
		\end{split}
	\end{equation}
	where $\check{C}:= 2C_4 \|\nabla
	(u-u_\tau)\varphi_R\|^{p-1}_{L^\infty(\R^N_+\cap\text{supp}(\psi))}$.
	Exploiting the weighted Young inequality with exponents
	$\displaystyle \left(\frac{\alpha+1}{\alpha}, \ \alpha+1\right)$ we
	obtain
	
	\begin{equation}\label{middlediseq2}
		\begin{split}
			&\alpha C_1 \int_{\R^N_+ \cap \text{supp}(\psi)} \left(|\nabla u| +
			|\nabla u_\tau|\right)^{p-2} |\nabla w_{\tau, \varepsilon}|^2
			w_{\tau, \varepsilon}^{\alpha-1} \varphi_R^2 \, dx\\
			& \leq  \check{C}\int_{\R^N_+ \cap \text{supp}(\psi)}  |\nabla
			\varphi_R| w_{\tau,\varepsilon}^{\alpha} \, dx\\
			&\leq \frac{\check{C}}{\sigma^{\alpha+1}(\alpha+1)}\int_{\R^N_+ \cap
				\text{supp}(\psi)} |\nabla \varphi_R|^{\alpha+1} \, dx\, +
			\frac{\check{C}\alpha}{\alpha+1}\sigma^{\frac{\alpha}{\alpha+1}}\int_{\R^N_+
				\cap \text{supp}(\psi)} w_{\tau,\varepsilon}^{\alpha+1} \, dx \,\\
			&\leq \frac{\dot{C}}{R^{\alpha-(N-2)}} +
			\frac{\check{C}\alpha}{\alpha+1}\sigma^{\frac{\alpha}{\alpha+1}}
			\int_{\R^N_+ \cap \text{supp}(\psi)}
			\left[w_{\tau,\varepsilon}^{\frac{\alpha+1}{2}}\right]^2 \, dx. \\
		\end{split}
    \end{equation}		

>From \eqref{middlediseq2} and exploiting the Poincar\'e inequality in the $x_N$-direction it follows that
	\begin{equation}\label{middlediseq3}
	    \begin{split}
	    	&\alpha C_1 \int_{\R^N_+ \cap \text{supp}(\psi)} \left(|\nabla u| +
	    	|\nabla u_\tau|\right)^{p-2} |\nabla w_{\tau, \varepsilon}|^2
	    	w_{\tau, \varepsilon}^{\alpha-1} \varphi_R^2 \, dx\\	
			&\leq \frac{\dot{C}}{R^{\alpha-(N-2)}} +
			\frac{\check{C}\alpha}{\alpha+1}\sigma^{\frac{\alpha}{\alpha+1}}
			\int_{B'_{2R}(0)} \left( \int_{\{y \leq k\}}
			\left[w_{\tau,\varepsilon}^{\frac{\alpha+1}{2}}\right]^2 \, dy
			\right) \, dx'\,\\
			&\leq\frac{\dot{C}}{R^{\alpha-(N-2)}} +
			\frac{\check{C}\alpha}{\alpha+1}\sigma^{\frac{\alpha}{\alpha+1}}
			C^2_P(k) \left(\frac{\alpha+1}{2}\right)^2  \int_{\R^N_+ \cap
				\text{supp}(\psi)} |\nabla w_{\tau, \varepsilon}|^2 w_{\tau,
				\varepsilon}^{\alpha-1}\, dx\,\\
			&\leq
			\frac{\check{C}\alpha}{\alpha+1}\sigma^{\frac{\alpha}{\alpha+1}}
			C^2_P(k) \left(\frac{\alpha+1}{2}\right)^2  \int_{\R^N_+ \cap
				\text{supp}(\psi)}\left(|\nabla u| + |\nabla
			u_\tau|\right)^{(2-p)+(p-2)} |\nabla w_{\tau, \varepsilon}|^2 w_{\tau,
				\varepsilon}^{\alpha-1}\, dx\,\\
			&+\frac{\dot{C}}{R^{\alpha-(N-2)}}
		\end{split}
	\end{equation}
	where $C_P$ is the Poincar\'e constant.
	Let us point out that, by \eqref{grow2233245555}, \eqref{XXX} and standard regularity theory \cite{DB,T}, it follows that
	\begin{equation}\label{jkgrhdkghkdjghskl}
		|\nabla u|+|\nabla
		u_\tau|\leq C\qquad\text{in}\,\,\, \text{supp}(w_{\tau, \varepsilon}) \subset
		\subset  \{\hat k\leq x_N \leq \hat{K}\}\,.
	\end{equation}
	Hence we
	have
	\begin{equation}\label{finaldiseq}
		\begin{split}
			& \int_{\mathcal{C}(R)} \left(|\nabla u| + |\nabla
			u_\tau|\right)^{p-2} |\nabla w_{\tau, \varepsilon}|^2
			w_{\tau, \varepsilon}^{\alpha-1} \, dx \\
			&\leq \vartheta \int_{\mathcal{C}(2R)} \left(|\nabla u| + |\nabla
			u_\tau|\right)^{p-2}|\nabla w_{\tau, \varepsilon}|^2 w_{\tau,
				\varepsilon}^{\alpha-1}\, dx\, + \frac{\dot{C}}{R^{\alpha-(N+2)}}
		\end{split}
	\end{equation}
	where $\displaystyle \mathcal{C}(R):= \left(\R^N_+ \cap (B'_R(0)\times\mathbb R)\right) $ and
	$\vartheta:=\frac{\check{C}\alpha}{\alpha+1}\sigma^{\frac{\alpha}{\alpha+1}}
	C^2_P(k) \left(\frac{\alpha+1}{2}\right)^2 \|(|\nabla u|+|\nabla
	u_\tau|)^{2-p}\|_\infty$. We set $$\displaystyle
	g(R):=\frac{\dot{C}}{R^{\alpha-(N+2)}}$$ and $$\displaystyle
	\mathcal{L}(R):=\int_{\mathcal{C}(R)} \left(|\nabla u| + |\nabla
	u_\tau|\right)^{p-2} |\nabla w_{\tau, \varepsilon}|^2 w_{\tau,
		\varepsilon}^{\alpha-1} \, dx$$
	so that
	\begin{equation}\nonumber
		\mathcal{L}(R)\leq \theta \mathcal{L}(2R)+g(R)\,.
	\end{equation}
	Now we fix
	$\alpha$ sufficiently large
	so that $g(R) \rightarrow 0$ as $n \rightarrow + \infty$ and,
	consequently, we take $\sigma$ small enough so that $\vartheta <
	2^{(\alpha-\gamma)\beta+1}$. This allows to exploit  Lemma 2.1 of \cite{FMS} it
	follows that $$\mathcal{L}(R)=0$$
	for any $R>0$.
	This proves that actually $w_{\tau, \varepsilon}$ is constant and therefore $w_{\tau, \varepsilon}=0$ since it vanishes near the boundary. This is a contradiction with \eqref{kgjklfgjfgjò} thus proving the result in the case $1<p<2$\,.

	\noindent {\bf {Case 2: $p \geq 2$.}} We set
	\begin{equation}\label{test}
		\psi:=w_{\tau, \varepsilon} \varphi_R^2
	\end{equation}
	with $w_{\tau,\varepsilon}$  and
	$\varphi_R$ defined as in the previous case $1<p<2$ . Arguing exactly
	as in the case $1<p<2$ we arrive to
	\begin{equation}\label{starteqpmin2}
		\begin{split}
			& \int_{\R^N_+ \cap \text{supp}(\psi)} (|\nabla u|^{p-2} \nabla u -
			|\nabla u_\tau|^{p-2} \nabla u_\tau, \nabla w_{\tau, \varepsilon})
			\varphi_R^2 \, dx \\
			=& -2 \int_{\R^N_+ \cap \text{supp}(\psi)} (|\nabla u|^{p-2} \nabla
			u - |\nabla u_\tau|^{p-2} \nabla u_\tau, \nabla \varphi_R)w_{\tau,
				\varepsilon}  \varphi_R \, dx\\
			& + \int_{\R^N_+ \cap \text{supp}(\psi)} \left( \frac{1}{u^\gamma} -
			\frac{1}{u_\tau^\gamma}\right)w_{\tau, \varepsilon} \varphi_R^2 \,
			dx
		\end{split}
	\end{equation}
	From \eqref{starteqpmin2}, using \eqref{eq:inequalities} and the
	Mean Value Theorem, we deduce that
	\begin{equation}\label{middlediseqjuhikjhgu}
		\begin{split}
			& C_1 \int_{\R^N_+ \cap \text{supp}(\psi)} \left(|\nabla u| +
			|\nabla u_\tau|\right)^{p-2} |\nabla w_{\tau, \varepsilon}|^2
			\varphi_R^2 \, dx \\
			&\leq  \int_{\R^N_+ \cap \text{supp}(\psi)} (|\nabla u|^{p-2} \nabla
			u - |\nabla u_\tau|^{p-2} \nabla u_\tau, \nabla w_{\tau,
				\varepsilon}) \varphi_R^2 \, dx\\
			&= -2 \int_{\R^N_+ \cap \text{supp}(\psi)} (|\nabla u|^{p-2} \nabla
			u - |\nabla u_\tau|^{p-2} \nabla u_\tau, \nabla \varphi_R) w_{\tau,
				\varepsilon} \varphi_R \, dx\\
			&+ \int_{\R^N_+ \cap \text{supp}(\psi)} \left( \frac{1}{u^\gamma} -
			\frac{1}{u_\tau^\gamma}\right) w_{\tau, \varepsilon}
			\varphi_R^2 \, dx\\
			&\leq 2C_2 \int_{\R^N_+ \cap \text{supp}(\psi)} \left(|\nabla u| +
			|\nabla u_\tau|\right)^{p-2} |\nabla w_{\tau, \varepsilon}|w_{\tau, \varepsilon} |\nabla
			\varphi_R| \varphi_R \, dx \\
			&- \gamma \int_{\R^N_+ \cap \text{supp}(\psi)}
			\frac{1}{\xi^{\gamma+1}} (u-u_\tau)w_{\tau,\varepsilon} \varphi_R^2
			\, dx
		\end{split}
	\end{equation}
	where $\xi$ belongs to $\overrightarrow{(u,u_\tau)}$. Exploiting the
	Young inequality to the right hand side we have
	\begin{equation}\label{middlediseqjuhikjhgsfdrgsfdsafu}
		\begin{split}
			& C_1 \int_{\R^N_+ \cap \text{supp}(\psi)} \left(|\nabla u| +
			|\nabla u_\tau|\right)^{p-2} |\nabla w_{\tau, \varepsilon}|^2
			\varphi_R^2 \, dx \\
			&\leq 2C_2 \int_{\R^N_+ \cap \text{supp}(\psi)} \left(|\nabla u| +
			|\nabla u_\tau|\right)^{p-2} |\nabla w_{\tau, \varepsilon}| \
			|\nabla \varphi_R|  w_{\tau, \varepsilon} \varphi_R \, dx \\
			&- \int_{\R^N_+ \cap \text{supp}(\psi)} \frac{1}{\xi^{\gamma+1}}
			(u-u_\tau)w_{\tau,\varepsilon} \varphi_R^2 \, dx\\
			&\leq \sigma C_2 \int_{\R^N_+ \cap \text{supp}(\psi)} \left(|\nabla
			u| +
			|\nabla u_\tau|\right)^{p-2} |\nabla w_{\tau, \varepsilon}|^2 \, dx \\
			&+ \frac{C_2}{\sigma} \int_{\R^N_+ \cap \text{supp}(\psi)}
			\left(|\nabla u| + |\nabla u_\tau|\right)^{p-2} |\nabla \varphi_R|^2
			w_{\tau,
				\varepsilon}^2 \varphi_R^2 \, dx\\
			&- \gamma \int_{\R^N_+ \cap \text{supp}(\psi)}
			\frac{1}{\xi^{\gamma+1}}
			(u-u_\tau)w_{\tau,\varepsilon} \varphi_R^2 \, dx\\
		\end{split}
	\end{equation}
	As above we shall exploit the fact that
	$|\nabla u|$ and $|\nabla u_\tau|$ are uniformly
	bounded in $\R^N_+ \cap \text{supp}(\psi)$, see \eqref{jkgrhdkghkdjghskl}. Therefore we get
	\begin{equation}\label{kjhkjhgkjh}
		\begin{split}
			& C_1 \int_{\R^N_+ \cap \text{supp}(\psi)} \left(|\nabla u| +
			|\nabla u_\tau|\right)^{p-2} |\nabla w_{\tau, \varepsilon}|^2
			\varphi_R^2 \, dx \\
			&\leq \sigma C_2 \int_{\R^N_+ \cap \text{supp}(\psi)} \left(|\nabla
			u| +
			|\nabla u_\tau|\right)^{p-2} |\nabla w_{\tau, \varepsilon}|^2 \, dx \\
			&+ \left(\frac{\check{C}}{\sigma R^2}-\dot{C}\right) \int_{\R^N_+
				\cap \text{supp}(\psi)} w_{\tau, \varepsilon}^2 \varphi_R^2 \, dx
		\end{split}
	\end{equation}
	
	where $\check{C}$ e $\dot{C}$ are positive constants. By taking $R_0
	>0$ sufficiently large it follows that $\displaystyle \frac{\check{C}}{\sigma
		R^2}-\dot{C}<0$ for every $R \geq R_0$. Hence we have
	
	\begin{equation}\label{kjhkjhgkjh}
		\begin{split}
			&\int_{\R^N_+ \cap \text{supp}(\psi)} \left(|\nabla u| + |\nabla
			u_\tau|\right)^{p-2} |\nabla w_{\tau, \varepsilon}|^2 \varphi_R^2 \,
			dx \\
			&\leq \frac{\sigma C_2}{C_1} \int_{\R^N_+ \cap \text{supp}(\psi)}
			\left(|\nabla u| + |\nabla u_\tau|\right)^{p-2} |\nabla w_{\tau,
				\varepsilon}|^2 \, dx.
		\end{split}
	\end{equation}
	As above, for
	$\displaystyle \mathcal{C}(R):= \left(\R^N_+ \cap (B'_R(0)\times\mathbb R)\right) $,  we set  $$\displaystyle
	\mathcal{L}(R):=\int_{\mathcal{C}(R)} \left(|\nabla u| + |\nabla
	u_\tau|\right)^{p-2} |\nabla w_{\tau, \varepsilon}|^2  \, dx$$
	so that
	\begin{equation}\nonumber
		\mathcal{L}(R)\leq \theta \mathcal{L}(2R)\,.
	\end{equation}
	where $\vartheta:=\frac{\sigma C_2}{C_1}>0$ is sufficiently small when $\sigma>0$ is
	sufficiently small. Applying again Lemma 2.1 of \cite{FMS} it follows
	that
	$$\int_{\mathcal{C}(R)} \left(|\nabla u| + |\nabla
	u_\tau|\right)^{p-2} |\nabla w_{\tau, \varepsilon}|^2 \, dx =0$$
	for any $R\geq R_ 0$.
	This provides a contradiction exactly as in the case $1<p<2$ so that the thesis follows also in the case $p\geq 2$.
\end{proof}

The 1D-symmetry result in Theorem \ref{1D} leads to the study of the one dimensional problem:
\begin{equation}\label{ODEeq}
\begin{cases} \displaystyle
-\left(|u'|^{p-2}u'\right)'=\frac{1}{u^\gamma}\, &  t \in \R^+ \\
u> 0 &   t \in \R^+   \\
u(0)=0
\end{cases}
\end{equation}
where $\gamma >1$ and $u \in C^{1,\alpha}(\R^+)\cap C(\R^+\cup
\{0\})$. As a consequence of the common feeling we expect uniqueness for such a problem, since the source term is decreasing. By the way the proof is not straightforward since the source term is decreasing but singular at zero. Now we present the following lemma, postponing the proof in the appendix.

\begin{lem}\label{derbehaviorODE}
Let $u\in C^{1,\alpha}(\R^+)\cap C(\R^+\cup \{0\})$ be a solution
to  \eqref{ODEeq}.  Assume that there exists a  positive
constant $C_u$  such that
\begin{equation}\label{inftybehavior}
\frac{t^\beta}{C_u}  \leq u(t) \leq  C_u t^\beta
\end{equation}
for $t$ sufficiently large and $\beta:= \frac{p}{\gamma+p-1}$. Then
there exists a positive constant $C'_u$  such that
\begin{equation}\label{inftybehaviorder}
\frac{t^{\beta-1}}{C'_u}  \leq u'(t) \leq  C'_u t^{\beta-1}
\end{equation}
for $t$ large enough.
\end{lem}

Now we are ready to prove our uniqueness result:
\begin{prop}\label{ODEuniqueness}
Problem
\eqref{ODEeq} admits a unique solution $u\in C^{1,\alpha}(\R^+)\cap C(\R^+\cup \{0\})$ satisfying \eqref{inftybehavior} given by
\begin{equation}\label{uniqueSolODE}
u(t)=Mt^\beta
\end{equation}
where  $\displaystyle M:=\left[
\frac{(\gamma+p-1)^p}{p^{p-1}(p-1)(\gamma-1)}\right]^{\frac{1}{\gamma+p-1}}$
and $\displaystyle \beta:=\frac{p}{\gamma+p-1}$.
\end{prop}

\begin{proof}
Arguing by contradiction we assume that there exist two positive
solution $u, v \in C^{1,\alpha}(\R^+)\cap C(\R^+\cup \{0\})$ to
problem \eqref{ODEeq} such that $u \not \equiv v$. Let us consider the  cutoff function $\varphi_R \in C^\infty_c(\R)$, $R>0$, such
that $\varphi_R(t)=1$ if $t \in [-R,R]$, $\varphi_R(t)=0$ if $t \in
(-\infty,-2R) \cup (2R,+\infty)$ and $|\varphi'(t)|< \frac{2}{R}$
for every $t \in (-2R,-R) \cup (R,2R)$. For
$\varepsilon>0$ (small) we set
\[
w_\varepsilon=(u-v-\varepsilon)^+\qquad\text{and}\qquad\psi:=[(u-v-\varepsilon)^+]^\alpha \varphi_R^2
\]
with $\alpha >0$ (large). Passing through the weak formulation of problem
\eqref{ODEeq} for $u$ and $v$, subtracting and using  standard elliptic estimates
estimates and \eqref{eq:inequalities} we obtain
\begin{equation}\label{ODEdiseq}
\begin{split}
&\alpha C_1 \int_0^{2R} \left(|u'| + |v'|\right)^{p-2}
|w'_\varepsilon|^2 w_\varepsilon^{\alpha-1} \varphi_R^2 \, dt \leq
\alpha \int_0^{2R} (|u'|^{p-2} u' - |v'|^{p-2} v',w'_\varepsilon)
w_ \varepsilon^{\alpha-1} \varphi_R^2 \, dt\\
&= -2 \int_R^{2R} (|u'|^{p-2} u' - |v'|^{p-2} v', \varphi'_R)
w_\varepsilon^\alpha \varphi_R \, dt + \int_0^{2R} \left(
\frac{1}{u^\gamma} - \frac{1}{v^\gamma}\right) w_\varepsilon^\alpha \varphi_R^2 \, dt\\
&\leq 2C_2 \int_R^{2R} \left(|u'| + |v'|\right)^{p-2}
|w'_\varepsilon| w_\varepsilon^\alpha |\varphi'_R| \varphi_R \, dt -
\gamma \int_R^{2R} \frac{1}{\xi^{\gamma+1}}
(u-v)w_\varepsilon^\alpha \varphi_R^2 \, dt
\end{split}
\end{equation}
with $\xi\in\overrightarrow{(u,v)}$. Exploiting the
weighted Young inequality to the right hand side we have
\begin{equation}\label{ODEdiseq2}
\begin{split}
&\alpha C_1 \int_0^{2R} \left(|u'| + |v'|\right)^{p-2}
|w'_\varepsilon|^2 w_\varepsilon^{\alpha-1} \varphi_R^2 \, dt\\
&\leq L C_2 \int_R^{2R} \left(|u'| + |v'|\right)^{p-2}
|w'_\varepsilon|^2
w_\varepsilon^{\alpha-1} \, dt\\
& + \frac{C_2}{L R^2}\int_R^{2R} \left(|u'| + |v'|\right)^{p-2}
w_\varepsilon^{\alpha+1}
\varphi^2_R \, dt \\
&- \gamma \int_R^{2R} \frac{1}{\xi^{\gamma+1}}
w_\varepsilon^{\alpha+1} \varphi_R^2 \, dt
\end{split}
\end{equation}
By Lemma \ref{derbehaviorODE}  it follows that
\begin{equation}\label{ODEdiseq3}
\begin{split}
&\alpha C_1 \int_0^{2R} \left(|u'| + |v'|\right)^{p-2}
|w'_\varepsilon|^2 w_\varepsilon^{\alpha-1} \varphi_R^2 \, dt\\
&\leq L C_2 \int_R^{2R} \left(|u'| + |v'|\right)^{p-2}
|w'_\varepsilon|^2 w_\varepsilon^{\alpha-1} \, dt\\
& + \frac{\check{C}}{L} \int_R^{2R} t^{(\beta-1)(p-2)-2}
w_\varepsilon^{\alpha+1} \varphi^2_R \, dt \\
&- \dot{C} \int_R^{2R} t^{-\beta(\gamma+1)}
w_\varepsilon^{\alpha+1} \varphi_R^2 \, dt\\
&\leq L C_2 \int_R^{2R} \left(|u'| + |v'|\right)^{p-2}
|w'_\varepsilon|^2 w_\varepsilon^{\alpha-1} \, dt\\
& +
\left(\frac{\check{C}}{L}-\hat{C}\right)\frac{1}{R^{\beta(\gamma+1)}}\int_R^{2R}
w_\varepsilon^{\alpha+1} \varphi^2_R \, dt
\end{split}
\end{equation}
where we also used the fact that $t/2\leq R\leq t$ when $t\in [R,2R]$.
Now we fix $L$ sufficiently large such that $(\frac{\check{C}}{L}-\hat{C})\leq 0$ so that
\begin{equation}\label{ODEdiseq4}
\begin{split}
\int_0^{R} \left(|u'| + |v'|\right)^{p-2} |w'_\varepsilon|^2
w_\varepsilon^{\alpha-1} \, dt \leq \frac{L C_2}{\alpha C_1}
\int_0^{2R} \left(|u'| + |v'|\right)^{p-2} |w'_\varepsilon|^2
w_\varepsilon^{\alpha-1} \, dt.
\end{split}
\end{equation}

Hence we define
$$\mathcal{L}(R):=\int_0^{R} \left(|u'| + |v'|\right)^{p-2} |w'_\varepsilon|^2
w_\varepsilon^{\alpha-1} \, dt.$$
By Lemma \ref{derbehaviorODE} we deduce that $\mathcal{L}(\cdot)$ has polynomial growth, namely
$$\mathcal{L}(R) \leq C R^{(\beta-1)(p-2)} R^{2(\beta-1)} R^{\beta(\alpha-1)} \int_0^{R} \, dt = CR^{(\beta-1)p+\beta(\alpha-1)+1}=CR^{\sigma}$$
with $\sigma:= (\alpha-\gamma)\beta+1$.
We take
$\alpha>0$ sufficiently large so that $\sigma>0$ and
$\frac{\vartheta C_2}{\alpha C_1} < 2^{-\sigma}$ so that Lemma 2.1 of \cite{FMS} apply and shows that
$$\mathcal{L}(R)=0.$$
>From this it follows that $u \leq v + \varepsilon$ for every
$\varepsilon>0$, hence $u \leq v$. Arguing in the same way it follows that $u \geq v$ and this proves the uniqueness result. To conclude the proof it is now sufficient to check that the function defined
in \eqref{uniqueSolODE} solve the problem.
\end{proof}
\begin{proof}[Proof of Theorem \ref{uniqueness}]
Once that Theorem \ref{1D} is in force, the proof of Theorem \ref{uniqueness} is a consequence of Proposition \ref{ODEuniqueness}.
\end{proof}

\section{Asymptotic analysis near the boundary and proof of Theorem \ref{Hopf}}\label{sec3}
We start this section considering the auxiliary problem:
\begin{equation}\label{auxprob}
\begin{cases} \displaystyle
-\Delta_p u =  \frac{p(x)}{u^\gamma} & \text{in}\,\, \mathcal{D} \\
u> 0 &  \text{in}\,\,\mathcal{D}
\end{cases}
\end{equation}
where $\mathcal{D}$ is a bounded smooth domain of $\R^N$, where $p
\in L^\infty(\mathcal{D})$ and $p(x)\geq c>0$ a.e. in $\mathcal{D}$,
$\gamma
>1$ and $u \in W^{1,p}_{loc}(\mathcal{D})\cap
C^0(\overline{\mathcal{D}})$.
For this kind of problems, generally, the weak comparison principle holds true. This is manly due to the monotonicity properties of the source term. In spite of this remark, the proof is not straightforward when considering sub/super solutions that are not smooth up the boundary. Therefore we provide here below a self contained proof of a comparison principle that we shall exploit later on.
\begin{lem}\label{wcpbd}
Let $u \in W^{1,p}_{loc}(\mathcal{D})\cap
C^0(\overline{\mathcal{D}})$ be a subsolution of problem
\eqref{auxprob} in the sense of \eqref{problem1weaksubsol} and let
$v\in W^{1,p}_{loc}(\mathcal{D})\cap C^0(\overline{\mathcal{D}})$ be
a supersolution of problem \eqref{auxprob} in the sense of
\eqref{problem1weaksupersol}. Then, if $u \leq v$ on $\partial
\mathcal{D}$ it follows that $u \leq v$ in $\mathcal{D}$.
\end{lem}

The proof of this lemma is contained in the appendix. We exploit now Lemma \ref{wcpbd} to study the boundary behaviour of the  solutions  to
\eqref{problem1}. The proof is actually the one in \cite{lazer}. Since we could not find an appropriate reference for the estimates that we need, we repeat the argument. We denote with $\phi_1$ the first (positive) eigenfunction of the $p$-laplacian in $\Omega$. Namely
\begin{equation}\label{einvaluprob}
\begin{cases} \displaystyle
-\Delta_p\phi_1=  \lambda_1 \phi_1^{p-1} & \text{in}\,\,\Omega  \\
\phi_1=0 & \text{on}\,\,\partial \Omega.
\end{cases}
\end{equation}

Having in mind Lemma \ref{wcpbd} we can prove a similar result to the one in \cite{lazer}, but in the quasilinear setting.

\begin{thm}\label{estbelowabove} Let  $u \in C^{1,\alpha}_{loc} (\Omega) \cap
C(\overline{\Omega})$ be a positive solution to
\eqref{problem1}.
Then there exist two positive constants $m_1$, $m_2$ and there
exists $\delta >0$ sufficiently small such that
\begin{equation} \label{estimateOnU}
m_1 \phi_1(x)^{\frac{p}{\gamma+p-1}} \leq u(x) \leq m_2
\phi_1(x)^{\frac{p}{\gamma+p-1}} \; \; \; \forall \; x \in
I_\delta(\partial \Omega).
\end{equation}
\end{thm}

We postpone the proof of Theorem \ref{estbelowabove} in the appendix. We are now ready to prove Theorem \ref{Hopf} exploiting the previous preliminary results.
\begin{proof}[Proof of Theorem \ref{Hopf}]
Since the domain is of class $C^{2,\alpha}$ we may and do reduce to
work in a neighborhood of the boundary $I_{\bar\delta}(\partial \Omega)$
where the \emph{unique nearest point property} holds (see e.g.
\cite{FOO}). Arguing by contradiction, let us assume that there
exists a sequence of points $\{x_n\}$ in $I_{\bar\delta}(\partial
\Omega)$, such that $x_n \longrightarrow x_0 \in
\partial \Omega$, as $n \rightarrow +\infty$, and
\begin{equation}\label{Hopfassurdo}
\partial_{\nu(x_n)}u(x_n) \leq 0, \; \text{with} \; (\nu(x_n),\eta(x_n))
\geq \beta > 0.
\end{equation}
Without loss of generality, we can assume that $x_0=0 \in
\partial \Omega$ and $\eta(x_n)=e_N$. This follows by the fact that the p-Laplace operator is invariant under isometries. More precisely, for each $n\in\mathbb{N}$, we can consider an isometry $T_n: \R^N
\longrightarrow \R^N$ with the above mentioned properties just
composing a translation and a rotation of the axes. This procedure
generates a  new sequence of points $\{y_n\}$, where $y_n:=T_nx_n$,
such that every $y_n \in \text{span}\langle e_N \rangle$ and $y_n
\longrightarrow 0$ as $n \rightarrow +\infty$. Setting
$u_n(y):=u(T_n^{-1}(y))$, it follows that
\begin{equation}-\Delta_p u_n =\frac{1}{u_n^\gamma} + f(u_n) \qquad \text{in}\,\,\Omega_n=T_n(\Omega). \end{equation}
Now we set
\begin{equation}\label{seqScaling}
w_n(y):=\frac{u_n(\delta_ny)}{M_n}
\end{equation}
where $\delta_n:=\text{dist}(x_n,\partial
\Omega)=\text{dist}(T_nx_n,0)$ and $M_n:=u_n(\delta_n e_N)=u(x_n)$.
It follows that $\delta_n \rightarrow 0$ as $n \rightarrow +\infty$
and
\begin{itemize}
\item[-] $w_n$ is defined in
$\displaystyle \Omega_n^*:=\frac{\Omega_n}{\delta_n}$.\\

\item[-] $w_n(e_N)=1$.\\

\item[-] $M_n \rightarrow 0$, as $n \rightarrow +\infty$.
\end{itemize}
It is easy to see that $w_n$ weakly satisfies
\begin{equation}\label{prescaling}
\begin{split}
-\Delta_p w_n =\frac{\delta_n^p}{M_n^{\gamma+p-1}}\left(\frac{1}{w_n(y)^\gamma} +
M_n^\gamma f(u_n(\delta_ny))\right) \qquad \text{in} \; \Omega^*_n.
\end{split}
\end{equation}

The key idea of the proof is to argue by contradiction exploiting a limiting profile, that we shall denote by $u_\infty$, which is a solution to a limiting problem in a half space. The contradiction will then follows applying the classification result in Theorem \ref{uniqueness}. Here below we develop this argument and we suggest to the reader to keep in mind that
 $f$ is bounded, the term $M_n^\gamma f(u_n(\delta_ny))$ will
vanish since $M_n$ goes to zero and $\displaystyle
\frac{\delta_n^p}{M_n^{\gamma+p-1}}$  is bounded as a consequence of
Theorem \ref{estbelowabove}. Therefore the expected  limiting equation is:
\begin{equation}\label{limitequation}
-\Delta_p w_\infty = \frac{\tilde{C}}{w_\infty^\gamma} \qquad
\text{in} \; \R^N_+\,.
\end{equation}
Let us provide the details needed to pass to the limit.  We claim
that:
\begin{itemize}
\item[-] $w_n \overset{C^{1,\alpha}}{\longrightarrow} w_\infty$, as $n \rightarrow
+\infty$, in any compact set $K$ of $\R^N_+$.\\
\item[-] $w_\infty \in C^{1,\alpha}(\R^N_+) \cap C(\overline{\R^N_+})$.\\
\item[-] $w_\infty=0$ on $\partial \R^N_+$\,.\\

\end{itemize}
To prove this let us consider a compact set $K\subset \R^N_+$. For $n\in\mathbb N$ large we can assume that
$K\subset I_{\bar\delta/\delta_n}(\partial \Omega^*_n)$ so that   Theorem
\ref{estbelowabove} can be exploited.

\emph{Claim 1.} We claim that $w_n(y) > 0$ for all $y \in K$ and for
 $n \in \N$ large.

Let $y \in K$. Hence, by  Theorem \ref{estbelowabove}
$$w_n(y):=\frac{u_n(\delta_n y )}{M_n} \geq L \frac{(\text{dist}(\delta_n y, \partial \Omega_n))^{\frac{p}{\gamma + p-1}}}{M_n}.$$
In particular, by the fact that $\text{dist}(\delta_n y, \partial
\Omega_n) \geq \overline{C} \delta_n$, it follows that
\begin{equation}\label{wnstaccatadazero} w_n(y)
\geq L \frac{(\overline{C}\delta_n)^{\frac{p}{\gamma + p-1}}}{M_n}
\geq C(K, \gamma, m_1)>0\,.
\end{equation}

\emph{Claim 2.} We claim that $w_n
\overset{C^{1,\alpha}}{\longrightarrow} w_\infty$, as $n \rightarrow
+\infty$, in any compact set $K$ of $\R^N_+$.

Since $\text{dist}(y, \partial \Omega^*_n) \leq C $ for every $y \in
K$, by Theorem \ref{estbelowabove} it follows that
\begin{equation}
\begin{split}
w_n(y) = \frac{u_n(\delta_n\,y)}{M_n} &\leq L m_2
\frac{\left[\text{dist}(\delta_n y, \partial
\Omega_n)\right]^{\frac{p}{\gamma+p-1}}}{M_n}\\
&= L m_2 \frac{\delta_n^\frac{p}{\gamma+p-1}\left[\text{dist}(y,
\partial \Omega^*_n)\right]^{\frac{p}{\gamma+p-1}}}{M_n}\\
&\leq L m_2  C^\frac{p}{\gamma+p-1}
\frac{\delta_n^\frac{p}{\gamma+p-1}}{M_n}\\
& \leq L  C^\frac{p}{\gamma+p-1} C(K,m_2).
\end{split}
\end{equation}

Hence
$$\|w_n\|_{L^\infty(K)} \leq C(K)$$
for any compact set $K$ of $\R^N_+$. By standard regularity theory
(see e.g. \cite{GT}) it follows  that $w_n$ is uniformly bounded in
$C^{1,\alpha}(K')$ for any compact set $K' \subset K$. Therefore, by Ascoli's Theorem, we can pass to the limit in any compact set and with $C^{1,\alpha'}$ convergence. Exploiting a standard diagonal process we can therefore define the limiting  function $w_\infty$ that turns out to be a solution to \eqref{limitequation} in the half space.
The fact that $\displaystyle \frac{\Omega_n}{\delta_n}$ leads to the
limiting domain $\R^N_+$ as  $n \rightarrow +\infty$ follows by
standard arguments that we omit.

It remains to verify the Dirichlet datum for the limiting profile
$w_\infty$. More precisely we have to show that $w_\infty=0$ on
$\partial \R^N_+$. By Theorem \ref{estbelowabove} it follows that
\begin{equation}\label{dirichletdatum}
\begin{split}
w_n(y) = \frac{u_n(\delta_n\,y)}{M_n} &\leq  L m_2 \frac{\left[\text{dist}(\delta_n y, \partial
\Omega_n)\right]^{\frac{p}{\gamma + p-1}}}{M_n}\\
&= L m_2 \frac{\delta_n^\frac{p}{\gamma + p-1}\left[\text{dist}(y,
\partial \Omega^*_n)\right]^{\frac{p}{\gamma + p-1}}}{M_n}\\
&\leq C(K,L,m_2,m_1) \left[\text{dist}(y,
\partial \Omega^*_n)\right]^{\frac{p}{\gamma + p-1}}    \quad \text{in}
\ \Omega^*_n.
\end{split}
\end{equation}

Since $\Omega_n^* \rightarrow \R^N_+$, as $n$ goes to $+\infty$, by
\eqref{dirichletdatum} and \eqref{wnstaccatadazero}, passing to the
limit we have that
\begin{equation} \label{growthwinf} 0 \leq w_\infty(y) \leq C(K,L,m_2,m_1)
\left[\text{dist}(y,
\partial \R_+^N)\right]^{\frac{p}{\gamma+p-1}}.
\end{equation}
In a similar fashion, and exploiting again Theorem \ref{estbelowabove}, we also deduce that

\begin{equation} \label{growthwinfx}  w_\infty(y) \geq C(K,L,m_2,m_1)
\left[\text{dist}(y,
\partial \R_+^N)\right]^{\frac{p}{\gamma+p-1}}.
\end{equation}
By  \eqref{growthwinf} it follows that $w_\infty(y)=0$ as claimed. Nevertheless, collecting \eqref{growthwinf} and \eqref{growthwinfx},
we deduce that
$w_\infty$ has the right asymptotic behaviour needed to apply
Theorem \ref{uniqueness}, see \eqref{grow223324}. This shows that
that $w_\infty$ is the unique solution to \eqref{limitequation} given by
\begin{equation}\label{limitsol}
w_\infty(x)=w_\infty(x_N) =\left[\frac{\tilde{C}}{p^{p-1}}
\frac{(\gamma+p-1)^p}{(p-1)(\gamma-1)}\right]^{\frac{1}{\gamma+p-1}}
(x_N)^{\frac{p}{\gamma+p-1}}.
\end{equation}
On the other hand, passing to the limit in \eqref{Hopfassurdo}, it would follows that
$$\partial_{\bar\nu} w_\infty(e_N) \leq 0 $$
for some $\bar\nu\in\mathbb{R}^N$ with $(\bar\nu,e_N)>0$. Clearly this is a contradiction with \eqref{limitsol} thus proving the result.
\end{proof}

Now using the Theorem \ref{Hopf} we want to prove the symmetry
result.

\section{Symmetry: Proof of Theorem \ref{symmetrybdd}}\label{sec4}
In this section we prove our symmetry (and monotonicity) result. Actually we provide the details needed for the application of the \emph{moving plane method}. For the semilinear case  see \cite{CES, masha, montrucchio}, in the quasilinear setting we use the technique developed in \cite{DSJDE}. \\

\noindent We start with some notation: for a real number $\lambda$ we set
\begin{equation}\label{eq:sn2}
\Omega_\lambda=\{x\in \Omega:x_1 <\lambda\}
\end{equation}
\begin{equation}\label{eq:sn3}
x_\lambda= R_\lambda(x)=(2\lambda-x_1,x_2,\ldots,x_n)
\end{equation}
which is the reflection through the hyperplane $T_\lambda :=\{x\in
\mathbb R^n :  x_1= \lambda\}$. Also let
\begin{equation}\label{eq:sn4}
a=\inf _{x\in\Omega}x_1.
\end{equation}
Finally we set
\begin{equation}\label{eq:sn33}
u_\lambda(x)=u(x_\lambda)\,.
\end{equation}
Finally we define
\begin{equation}\nonumber
\Lambda_0=\{a<\lambda<0 : u\leq
u_{t}\,\,\,\text{in}\,\,\,\Omega_t\,\,\,\text{for all
$t\in(a,\lambda]$}\}\,.
\end{equation}
In the following the critical set of $u$
$$Z_u:=\{\nabla u=0\}$$
will play a crucial role. Let us first note that, as a consequence of Theorem
\ref{Hopf}, we know that
\[
Z_u\subset\subset\Omega\,.
\]
This fact allows to exploit the results of \cite{DSJDE} since the solution is positive in the interior of the domain (and the nonlinearity is no more singular there). Therefore we conclude that
\[
|Z_u|=0\qquad\text{and}\qquad \Omega\setminus Z_u\quad\text{is connected}.
\]

\begin{proof}[Proof of Theorem \ref{symmetrybdd}]
The proof follows via the \emph{moving plane technique}. We start showing that:
 $$\Lambda_0 \neq \emptyset\,.$$
 To prove this, let us consider $\lambda>a$ with $\lambda-a$ small.
 By Theorem \ref{Hopf} it follows that
$$\frac{\partial u}{\partial x_1} > 0 \quad \text{in} \quad
\Omega_\lambda \cup R_\lambda(\Omega_\lambda),$$ and this immediately proves that $u <
u_\lambda$ in $\Omega_\lambda$.

Now we define $$\lambda_0 := \sup \Lambda_0.$$
We shall  show that $u \leq u_\lambda$ in $\Omega_\lambda$
for every $\lambda \in (a,0]$, namely that:
 $$\lambda_0 = 0\,.$$
 To prove this, we assume that
$\lambda_0<0$ and we reach a contradiction by proving that $u\leq
u_{\lambda_0+\nu}$ in $\Omega_{\lambda_0+\nu}$ for any
$0<\nu<\bar\nu$ for some  $\bar\nu>0$ (small). By continuity we know
that $u\leq u_{\lambda_0}$ in $\Omega_{\lambda_0}$.
The \emph{strong comparison principle} (see e.g. \cite{pucser,V}) holds true in $\Omega_{\lambda_0}\setminus Z_u$, providing that
\[
u< u_{\lambda_0}\qquad\text{in}\quad\Omega_{\lambda_0}\setminus Z_u\,.
\]
Note in fact that, in each connected component $\mathcal C$ of $\Omega_{\lambda_0}\setminus Z_u$, the \emph{strong comparison principle}
implies that $u< u_{\lambda_0}$ in $\mathcal C$ unless $u\equiv u_{\lambda_0}$ in $\mathcal C$. Actually the latter case is not possible. In fact, if $\partial\mathcal C\cap\partial\Omega\neq\emptyset$ this is not possible in view of the zero Dirichlet baundary datum since $u$ is positive in the interior of the domain. If else $\partial\mathcal C\cap\partial\Omega=\emptyset$ then we should have a \emph{local symmetry region} causing $\Omega\setminus Z_u$ to be not connected, against what we already remarked above.\\

\noindent Therefore, given a
compact set $K \subset \Omega_{\lambda_0}\setminus Z_u$, by uniform continuity we
can ensure that $u < u_{\lambda_0 + \nu}$ in $K$ for any $0 < \nu <
\bar \nu$ for some small $\bar \nu > 0$. Moreover, by Theorem
\ref{Hopf} and taking into account the zero Dirichlet boundary datum, it is easy to show that, for some $\delta>0$, we have that
 \begin{equation}\label{dkjfjla}
 u < u_{\lambda_0+\nu}\quad \text{in}\quad
I_\delta(\partial \Omega) \cap \Omega_{\lambda_0 + \nu}
\end{equation}
 for any $0< \nu < \bar \nu$.
 This is quite standard once that Theorem
\ref{Hopf} is in force. The hardest part is the study in the region near $\partial\Omega \cap T_{\lambda_0 + \nu}$. Here we exploit the monotonicity properties of the solutions proved in Theorem
\ref{Hopf} that works once we note that  $(e_1, \eta (x))>0$ in a neighborhood of $\partial\Omega \cap T_{\lambda_0 + \nu}$ since the domain is smooth and strictly convex.

Now we define
$$w_{\lambda_0+\nu}:= (u-u_{\lambda_0+\nu})^+$$
for any $0 < \nu < \bar \nu$. We already showed in \eqref{dkjfjla} that
$\text{supp}(w_{\lambda_0+\nu}) \subset \subset
\Omega_{\lambda_0+\nu}$. Moreover  $w_{\lambda_0+\nu}=0$
in $K$ by construction.

For any $\tau>0$ fixed, we can choose $\bar\nu$ small and $K$ large so that
\[
|\Omega_{\lambda_0+\nu}\setminus
K| < \tau\,.
\]
Here we are also exploiting the fact that the critical set $Z_u$ has zero Lebesgue measure (see \cite{DSJDE}).

In particular we take  $\tau$ sufficiently small so that the \emph{weak comparison principle in small
domains}  (see \cite{DSJDE}) works, showing that
$$w_{\lambda_0+\nu}=0 \quad \text{in} \quad \Omega_{\lambda_0 + \nu}$$
for any $0 < \nu < \bar \nu$ for some small $\bar \nu > 0$. But this
is in contradiction with the definition of $\lambda_0$. Hence
$\lambda_0=0$.

The desired symmetry (and monotonicity) result follows now
performing the procedure in the same way but in the opposite direction.
\end{proof}

\section*{Appendix}

We start this appendix, by proving Lemma \ref{derbehaviorODE} that is an essential tool in the ODE analysis.

\begin{proof}[Proof of Lemma \ref{derbehaviorODE}]	
	We first claim  that $u'(t)\geq0$ for every $t>0$. To prove this fact we argue by
	contradiction and  assume that there exist $t_0\geq0$
	such that $u'(t_0)<0$. Setting
	$$w(t):=|u'(t)|^{p-2}u'(t)$$ it
	follows by the equation in \eqref{ODEeq} that $w$ is a strictly
	decreasing function. Therefore
	$u'(t)\leq-C:=u'(t_0)<0$ for every $t \geq t_0$ and
	\begin{equation}\label{fhfgakfgkfh}
	u(t)=u(t_0)+\int_{t_0}^tu'(s) \, ds \leq u(t_0)-\int_{t_0}^t C \, ds = -Ct+Ct_0+u(t_0).
	\end{equation}
	This would force $u$ to be negative for $t$ large
	in contradiction with with the fact that $u$ is
	positive by assumption.  Therefore we deduce that $u'(t),w(t)\geq0$ for t
	sufficiently large. Recalling that $w$ is a strictly
	decreasing function, we deduce that actually $u'(t),w(t)>0$. Furthermore
	$w(t) \rightarrow M\geq 0$ as $t$ goes to
	$+\infty$.  I is easy to show that $M=0$.
	If $M > 0$ in fact, arguing as in \eqref{fhfgakfgkfh},  we would have
	$$u(t)\geq Mt+c$$
	for $t$ sufficiently large. This gives a contradiction with our
	initial assumption \eqref{inftybehavior}, hence $M=0$. \\
	
	\noindent Let us
	now set
	\[
	h(t):=\frac{t^{-\beta \gamma+1}}{\beta \gamma-1}\,.
	\] By Cauchy's
	Theorem we have that for $t$ large enough and $k>t$ fixed there
	exists $\xi_t \in (t,t+k)$ such that
	\begin{equation}\label{cauchylagrange}
	\frac{w(t)-w(t+k)}{h(t)-h(t+k)}=\frac{w'(\xi_t)}{h'(\xi_t)}.
	\end{equation}
	Letting $k \rightarrow +\infty$ in \eqref{cauchylagrange} we obtain
	\begin{equation}\label{cauchylagrangexfgdg}
	\frac{w(t)}{h(t)}=-\frac{([u']^{p-1})'(\xi_t)}{(\xi_t)^{-\beta\gamma}}=\frac{t^{\beta
			\gamma}}{u^\gamma}.
	\end{equation}
	for $t$ large enough. By \eqref{inftybehavior} and
	\eqref{cauchylagrangexfgdg} we deduce that $\frac{w(t)}{h(t)}$ is bounded at infinity, thus proving
	\eqref{inftybehaviorder}.
\end{proof}

Now we are ready to prove a useful weak comparison principle in bounded domains (Lemma \ref{wcpbd}).

\begin{proof}[Proof of Lemma \ref{wcpbd}]
	Let us set:
	\begin{equation}\label{testwcpbd}
	w_\varepsilon:=(u-v-\varepsilon)^+
	\end{equation}
	where $\varepsilon > 0$. We notice that $w_\varepsilon$ is suitable as
	test function since $\text{supp}(w_\varepsilon)\subset \subset
	\mathcal{D}$ and $u,v \in W^{1,p}_{loc}(\mathcal{D})$.
	Hence $w_\varepsilon \in W^{1,p}_0(\mathcal{D})$ and, by density arguments,  we can plug $w_\varepsilon$ as test function in
	\eqref{problem1weaksubsol} and \eqref{problem1weaksupersol} and by
	subtracting we obtain
	\begin{equation}\label{weaksubtranct}
	\begin{split}
	\int_{\mathcal{D} \cap \text{supp}(w_\varepsilon)} \left( |\nabla
	u|^{p-2}\nabla u - |\nabla v|^{p-2} \nabla v, \nabla w_\varepsilon
	\right) \, dx \, &\leq \int_{\mathcal{D} \cap
		\text{supp}(w_\varepsilon)} p(x) \left(\frac{1}{u^\gamma}-
	\frac{1}{v^\gamma} \right)w_\varepsilon \, dx.
	\end{split}
	\end{equation}
	Taking into account the fact that $u-v \geq u-v-\varepsilon$, the fact that $p(\cdot)$ is positive and $u^{-\gamma}$ is decreasing,  it follows that
	\begin{equation}\label{finalestlz}
	\int_{\mathcal{D} \cap \text{supp}(w_\varepsilon)} \left( |\nabla
	u|+ |\nabla v|\right)^{p-2}|\nabla w_\varepsilon|^2 \, dx \leq 0.
	\end{equation}
	By Fatou Lemma, as $\varepsilon$ tends to zero, we deduce that
	
	\begin{equation}\nonumber
	\int_{\mathcal{D} } \left( |\nabla
	u|+ |\nabla v|\right)^{p-2}|\nabla (u-v)^+|^2 \, dx \leq 0
	\end{equation}
	showing that $(u-v)^+$ is constant, and therefore zero by the boundary data. Thus we deduce that $u
	\leq v$ in $\mathcal{D}$ proving the thesis.
\end{proof}

Now we are ready to prove the analogous of the result contained in \cite{lazer}, but in the quasilinear setting.

\begin{proof}[Proof of Theorem \ref{estbelowabove}]
	We rewrite  the equation in \eqref{problem1} as
	\begin{equation}\label{racc}
	-\Delta_p u = \frac{1}{u^\gamma}+f(u) = \frac{p(x)}{u^\gamma} \quad
	\quad \text{in} \; \Omega
	\end{equation}
	where $p(x):=1+u(x)^\gamma f(u(x))$. In the following we assume that $\delta$ is small enough so that
	$$p(x)>0 \quad \quad \forall \, x \in I_\delta(\partial \Omega).$$
	Arguing as in \cite{lazer}, we exploit the principal eigenfunction
	$\phi_1$ of problem \eqref{einvaluprob} and the fact that $\phi_1 \in
	C^{1,\alpha} (\overline \Omega) $ (see e.g. \cite{anane, lindqv,
		klin}) and
	$$\nabla \phi_1 (x) \neq 0 \quad \quad \forall \, x \in \partial \Omega.$$
	For $\displaystyle t:=\frac{p}{\gamma+p-1}$ we set
	$\Psi:=s\,\phi_1^t$,  $s>0$. It is easy to see that
	$$-\Delta_p \Psi\,=\,\frac{g(x,s)}{\Psi^\gamma} \quad \quad \text{in}\,\,\, I_\delta(\partial \Omega)$$
	where
	\begin{equation} \label{gggg} g(x,s):=s^{\gamma+p-1} t^{p-1} \left[
	\frac{(\gamma-1)(p-1)}{\gamma+p-1} |\nabla v_1(x)|^p +  \lambda_1
	v_1(x)^p \right ].
	\end{equation}
	Since $0<t<1$,  we can choose
	two positive constants $s_1$ and $s_2$ such that $0<s_1<s_2$ and
	\begin{equation} \label{relation} g(x,s_1)<p(x)<g(x,s_2) \quad
	\quad \forall \ x \in I_\delta(\partial \Omega).
	\end{equation}
	Hence, setting $u_1:=s_1\phi_1^t$ and
	$u_2:=s_2\phi_1^t$, we have that
	\begin{equation}\label{supersol}
	-\Delta_p u_1 < \frac{p(x)}{u_1^\gamma} \quad \quad \text{in} \;
	I_\delta(\partial \Omega)
	\end{equation}
	and
	\begin{equation}\label{subsol}
	-\Delta_p u_2 > \frac{p(x)}{u_2^\gamma} \quad \quad \text{in} \;
	I_\delta(\partial \Omega)
	\end{equation}
	In order to control the datum on the boundary of $I_\delta(\partial \Omega)$ (in the interior of the domain), we need to switch from $u$ to
	$u_\beta:=\beta u$ observing that
	$$-\Delta_p u_\beta\,=\, \beta^{\gamma+p-1}\frac{p(x)}{u_\beta^\gamma}.$$
	For $\beta_1>0$  large it follows that
	$u_{\beta_1}$ and $u_1$ satisfy the following problem:
	\begin{equation}\label{supersolprob}
	\begin{cases} \displaystyle
	-\Delta_p u_{\beta_1} \geq \frac{p(x)}{u_{\beta_1}^\gamma} & \text{in}\,\,I_\delta(\partial \Omega) \\
	\displaystyle-\Delta_p u_1 < \frac{p(x)}{u_1^\gamma} & \text{in}\,\,I_\delta(\partial \Omega)\\
	\displaystyle u_{\beta_1} \geq u_1 & \text{on}\,\,\partial
	I_\delta(\partial \Omega).
	\end{cases}
	\end{equation}
	By Lemma \ref{wcpbd}
	it follows now that
	\begin{equation}\label{below}
	u_{\beta_1}=\beta_1 u \geq u_1 \quad \quad
	\text{in}\,\,I_\delta(\partial \Omega).
	\end{equation}
	Similarly, for$\beta_2>0$  small, it follows that
	$u_{\beta_2}$ and $u_2$ satisfy the problem:
	\begin{equation}\label{ubol}
	\begin{cases} \displaystyle
	-\Delta_p u_{\beta_2} \leq \frac{p(x)}{u_{\beta_2}^\gamma} & \text{in}\,\,I_\delta(\partial \Omega) \\
	\displaystyle-\Delta_p u_2 > \frac{p(x)}{u_2^\gamma} & \text{in}\,\,I_\delta(\partial \Omega)\\
	\displaystyle u_{\beta_2} \leq u_2 & \text{on}\,\,\partial
	I_\delta(\partial \Omega).
	\end{cases}
	\end{equation}
	By Lemma \ref{wcpbd} it follows that
	\begin{equation}\label{above}
	u_{\beta_2}=\beta_2 u \leq u_2 \quad \quad
	\text{in}\,\,I_\delta(\partial \Omega).
	\end{equation}
	Hence the thesis is proved with  $\displaystyle m_1:=\frac{s_1}{\beta_1}$ and
	$\displaystyle m_2:=\frac{s_2}{\beta_2}$.
\end{proof}

\end{document}